\newcommand{\comment}[1]{}
\newcount\hh
\newcount\mm
\mm=\time
\hh=\time
\divide\hh by 60
\divide\mm by 60
\multiply\mm by 60
\mm=-\mm
\advance\mm by \time
\def\hhmm{\number\hh:\ifnum\mm<10{}0\fi\number\mm}

\documentclass{article}

\usepackage{amssymb, amsthm,amsmath}
\usepackage{pictexwd, dcpic}
\usepackage[]{hyperref}
\hypersetup{
colorlinks=true,
linkcolor=black,
}
\usepackage{a4wide} 
\usepackage{fancyhdr}

\theoremstyle{definition} 
\newtheorem{theorem}{Theorem}
\newtheorem{prop}[theorem]{Proposition}
\newtheorem{lemma}[theorem]{Lemma}
\newtheorem*{defi}{Definition}
\newtheorem*{exam}{Example}
\newtheorem*{ack}{Acknowledgement}
\newtheorem{cor}[theorem]{Corollary}

\newtheorem{cons}[theorem]{Construction}
\theoremstyle{remark} 
\newtheorem*{nota}{Notation}

\newcommand\brac[1]{\langle #1 \rangle}
\newcommand\sep{\,:\,}
\newcommand\st {\mathrel{\ooalign{$\,\backepsilon$\cr\lower .7pt\hbox{\kern 1pt$-\,$}}}}
\newcommand\qf[1]{\mathrm{Quot}(#1)}	
\newcommand\spec{\mathrm{Spec}\,}	
\newcommand\sper{\mathrm{Sper}\,}	
\newcommand\supp{\mathrm{supp}}	
\newcommand\id{\mathrm{id}}	

\newcommand\Hom{\mathrm{Hom}}

\newcommand\R{\mathbb{R}}

\newcommand{\N}{\mathbb{N}}
\newcommand{\p}{\mathfrak{p}}
\newcommand\wo[1]{\backslash{\{#1\}}}
\newcommand{\bij}{\rightarrowtail \hspace{-1.9ex} \rightarrow}

\newcommand*{\lhrarrow}{\ensuremath{\lhook\joinrel\relbar\joinrel\rightarrow}}

\newcommand*{\lthrarrow}{\ensuremath{\relbar\joinrel\twoheadrightarrow}}
\renewcommand\mod{\,\mathrm{mod}\,}
\newcommand{\ic}{\mathrm{ic}}

\def\tboxit#1#2{{\setbox0=\hbox{\kern5pt#1\kern5pt}\edef\titlewidth{\the\wd0}%
                \setbox2=\vbox{#2} 
                \setbox2=\vbox{%
                    \vbox to1pt{\vss\hbox 
to\wd2{\strut\hfil#1\hfil}\vskip0pt}%
                    \box2 
                    } 
                \Tboxit{\titlewidth}{\Tboxit{\titlewidth}{\box2}}}} 

\def\Tboxit#1#2{\vbox{%
    \setbox0=\hbox{\vrule\kern3pt\vbox{\kern3pt#2\kern3pt}\kern3pt\vrule}%
    \hbox to\wd0{\hrulefill\kern#1\hrulefill}\nointerlineskip 
    \box0 
    \hrule 
    }}

\comment{
\makeatletter
\def\url@leostyle{%
  \@ifundefined{selectfont}{\def\UrlFont{\sf}}{\def\UrlFont{\small\ttfamily}}}
\makeatother
}

\begin{document}

\title{Uniqueness of real closure $*$ of regular rings}
\makeatletter
\let\mytitle\@title
\makeatother
\author{Jose Capco \\
\href{mailto:capco@fim.uni-passau.de}{\small{capco@fim.uni-passau.de}}\\
\emph{\small{Universit\"at Passau, Innstr. 33, 94032 Passau, Germany}} }
\date{}
\thispagestyle{empty}
\maketitle

\pagestyle{fancy}
\fancyhead[R]{\mytitle}
\fancyhead[L]{J. Capco}
\tolerance=500

\begin{abstract}
In this paper we give a characterisation of real closure $*$ of regular rings, which is quite similar to the characterisation
of real closure $*$ of Baer regular rings seen in \cite{Capco2}. We also characterize Baer-ness of regular rings using \emph{near-open} 
maps. The last part of this work will concentrate on classifying the real closure $*$ of Baer and non-Baer regular rings (upto isomorphisms) 
using continuous sections of the support map, we construct a topology on this set for the Baer case. For the case of non-Baer regular rings, 
it will be shown that almost no information of the ring structure of the Baer hull is necessary
in order to study the real and prime spectra of the Baer hull. We shall make use of the absolutes of Hausdorff spaces in order to give a construction
of the spectra of the Baer hulls of regular rings. 
Finally we give example of a Baer regular ring that is not rationally complete.
\begin{description}
\item[Mathematics Subject Classification (2000):] Primary 13J25; Secondary 06E15, 16E50
\item[Keywords:] real closed $*$ rings, Baer von Neumann regular rings, absolutes of Hausdorff spaces, rational completeness, continuous sections, 
near open maps, compact-open topology, point convergence topology, Gleason spaces.
\end{description}
\end{abstract}

\footnote{Supported by Deutsche Forschungsgemeinschaft and Universit\"at Passau.}

Henceforth, when we say \emph{regular ring}, we mean a von Neumann regular ring. When we say ring, we usually mean commutative unitary 
partially ordered ring.
Poring is a ring $A$ that has a partial ordering $A^+$. 

We assume that the reader is familiar with the notations used in \cite{Capco} and \cite{Capco2}. However for completeness, here are a list 
of notations that may be used.
\begin{nota}
Let $A$ be a ring and $x\in A$
\begin{itemize}
\item If $A$ is a poring then $\sper A$ is the topological space (Harrison Topology) consisting of prime cones containing $A^+$
\item $E(A):=\{e\in A \sep e^2 =e\}$ is the set of the idempotents of $A$
\item $B(A)$ is the Baer hull of $A$, if $A$ is a poring with partial ordering $A^+$ then we use the partial ordering
$$B(A)^+:=\{\sum_{i=1}^n b_i^2a_i \sep n\in \N, b_i\in B(A), a_i\in A^+ \textrm{ for } i=1,\dots,n\} $$
for $B(A)$
\item $Q(A)$ will be the complete ring of quotients of $A$. If $A$ is a poring with partial ordering $A^+$, then we use the partial 
ordering
$$Q(A)^+:=\{\sum_{i=1}^n x_i^2a_i \sep n\in \N, x_i\in Q(A), a_i\in A^+ \textrm{ for } i=1,\dots,n\} $$
for $Q(A)$
\item $\mathbf{CRings}$ is the category of commutative unitary rings with the usual ring homomorphisms (i.e. 1 is mapped to 1)
\item $D_A(x):=\{\p\in \spec A : x\not\in \p\}$, if it is clear with what rings we are dealing with we write $D(x)$ instead.
\item If $A$ is a poring $P_A(x):=\{\alpha\in \sper A : x \in \alpha\backslash \supp(\alpha)\}$, we may also write $P(x)$.
\item Let $\alpha\in\sper A$ then by $\rho(\alpha)$ we mean the real closed field (upto $A/\supp(\alpha)$-isomorphism) that is 
algebraic over $\qf{A/\supp(\alpha)}$ and such that $\alpha/\supp(\alpha)$ is positive in it.
\end{itemize}
\end{nota}

First a few note about \cite{Capco2}. There we constantly made use of a certain Theorem by Storrer that involved essential extension
of rings, but we made use of a rather stronger statement of the original Theorem (which is also true). The original Theorem found in \cite{Storrer} Statz 10.1 
states that if $A$ is a semiprime ring and if $B$ is an essential extension 
of $A$, then there exists a monomorphism of rings $Q(A)\hookrightarrow Q(B)$. But when one looks at the proof of Storrer's Theorem (which we 
shall officially call the \emph{Storrer's Satz}) one has more to say. In fact it was first pointed out by Raphael, in \cite{raphael} 
Theorem 3.12, that Storrer's Satz can be strengthened in the following way $\dots$

\begin{theorem} (Storerr's Satz) Let $A$ be a semiprime ring and let $B$ be an essential extension of $A$. Then there exists a monomorphism of 
rings $f: Q(A)\rightarrow Q(B)$ such that the diagram below commutes (in the category $\mathbf{CRings}$)
$$
\begindc{\commdiag}[1]
\obj(0,50)[A]{$A$}
\obj(0,0)[B]{$B$}
\obj(100,50)[C]{$Q(A)$}
\obj(100,0)[D]{$Q(B)$}
\mor{A}{B}{}[1,3]
\mor{B}{D}{}[1,3]
\mor{A}{C}{}[1,3]
\mor{C}{D}{$f$}[1,3]
\enddc
$$
where the unlabeled maps in the commutative diagram above are all canonical maps.
\end{theorem}

The proof of the above Theorem is omitted as it is already manifest in the proof of the original Theorem 
made by Storrer (\cite{Storrer} Satz 10.1). I have already made several use of this new form of the Theorem in my paper \cite{Capco2}.
This form of Storrer's Satz will be used very often in the future as well. By the way we assumed the partial orderings of our 
complete ring of quotients (i.e. they have the weakest partial ordering such that they contain the partial ordering of the original poring) we
at once see that Storrer's Satz also holds in the category of porings. That is, we can assume our rings to be porings and our 
ring homomorphisms to be poring morphisms.

\begin{cons}
For completeness, we shall write down how the monomorphism in the Storrer's Satz above is constructed.

For any ring $A$, there is a ring monomorphism $A\hookrightarrow Q(A)$ (see \cite{lambek} \S2.3 Proposition 6 p.40). We may also write 
$$Q(A)=\stackrel{\textstyle .}{\bigcup_{D\lessdot A}}\Hom_A(D,A)/\sim_A$$
where $\sim_A$ is a specific equivalence relation and $D\lessdot A$ means that $D$ \emph{is a dense ideal of A}. For readers unfamiliar with the terminology and concept used in the study
of the complete ring of quotients of rings, I suggest \cite{roqrof} \S1 and \cite{lambek} \S2.3 and \S2.4 p.36-46 as reference.

Henceforth, for any ring $A$ and for any $\phi\in \stackrel{\textstyle .}{\bigcup}\Hom_A(D,A)$ we write $[\phi]_A$ to mean the canonical
image of $\phi$ in $Q(A)$.

Now we are ready to make the construction. Let $A$ and $B$ satisfy the condition of the Storrer's Satz. Let $\phi: D\rightarrow A$ be
a module morphism with $D$ a dense ideal of $A$. Storrer showed the following
\begin{enumerate}
\item There is a maximal family $\{d_i\}_I \subset D$ such that $\oplus_I d_iA$ is a direct sum and is dense in $A$
\item $\overline D := \oplus_I d_i B$ is then a direct sum and is dense in $B$
\item We then associate $[\phi]_A$ to $[\overline \phi ]_B$  where
$$\overline \phi := \oplus_I \phi_i : \overline D \longrightarrow B$$
with $\overline \phi_i : d_iB\rightarrow B$ defined by $\phi_i(d_i):=\phi(d_i)\in A\subset B$. This association turns out to be not only a 
well-defined function between $Q(A)$ and $Q(B)$, but also a ring monomorphism satisfying the Storrer's Satz above.
\end{enumerate}
\begin{flushright}$\blacksquare$\end{flushright}
\end{cons}

There is another result by Raphael which I have made use in \cite{Capco2} and I will also make constant use of it hereafter. The result 
I shall call \emph{Raphael's Lemma} whose proof is a combination of proofs found (but not formally stated) in \cite{raphael} 
Lemma 1.14, Proposition 1.16 and Remark 1.17.

\begin{lemma} (Raphael's Lemma) If $A$ is a regular Baer ring and $B$ is a regular ring which is an essential extension of $A$ then 
$B$ is also Baer and we have a canonical homeomorphism 
$$\phi : \spec B \rightarrow \spec A\qquad \p \mapsto \p \cap A$$ 
whose inverse is 
$$\phi^{-1} : \spec A \rightarrow \spec B \qquad \mathfrak q \mapsto \mathfrak q B$$
\end{lemma}

\begin{lemma} \label{rcs_of_BHull} Let $A$ be a real regular ring and let $C$ be a real closure $*$ of $A$, then
\begin{enumerate}
\item $C$ can be regarded as a real closure $*$ of $B(A)$
\item The spectral map $\spec C\rightarrow \spec B(A)$ induced from 1. is a homeorphism.
\end{enumerate}
\end{lemma}
\begin{proof}
By Storrer's Satz, we have the following commutative diagram of rings
$$
\begindc{\commdiag}[1]
\obj(0,50)[A]{$A$}
\obj(0,0)[B]{$C$}
\obj(100,50)[C]{$Q(A)$}
\obj(100,0)[D]{$Q(C)$}
\mor{A}{B}{}[1,3]
\mor{B}{D}{}[1,3]
\mor{A}{C}{}[1,3]
\mor{C}{D}{}[1,3]
\enddc
$$
We can thus regard all the given rings as subrings of $Q(C)$. By Theorem 15 of \cite{Capco} we know that $C$ is Baer, thus by Proposition 
2 in \cite{Capco} $C$ contains all the idempotents of $Q(C)$. Specifically, $C$ contains $A$ and all the idempotents of $Q(A)$. But $A$ 
and the idempotents of $Q(A)$ together generate $B(A)$. Therefore $B(A)$ may indeed be regarded as a subring of $C$. We originally 
had $B(A)^+$ constructed in such a way that it is the partial ordering of $B(A)$ which is the weakest extension of $A^+$ (see \cite{Brum}
\S1.3 p.34-35). Thus $C^+\cap B(A)\supset B(A)^+\supset A^+$ and 
therefore $B(A)$ can in fact be regarded as a subporing of $C$. We thus have the following extension of porings
$$A\lhrarrow B(A)\lhrarrow C$$
we also know that $C$ is an integral and essential extension of $A$ meaning that it is also an integral and essential extension of $B(A)$.
$C$ being real closed $*$ implies that $C$ is indeed a real closure $*$ of $B(A)$. By Raphael's Lemma, 
$\spec C\rightarrow \spec B(A)$ is a homeomorphism.
\end{proof}

\begin{lemma}\label{Baerprime}
Let $A$ be a real regular ring and let $B,C$ be two real closure $*$ of $A$ such that they are not $A$-isomorphic. Then there exists
$\p\in \spec B(A)$ such that 
$$B/\p B \not\cong_{A/\p\cap A} C/\p C$$
\end{lemma}
\begin{proof}
Set $X:=\spec A, Y:=\spec B$ and $Z:=\spec C$. By Lemma \ref{rcs_of_BHull}, we regard $B(A)$ as a subporing of both $B$ and $C$ and we know
then that $\spec B$ and $\spec C$ are (canonically) homeomorphic to $\spec B(A)$. By Theorem 8 in \cite{Capco2} there is an $x\in X$  such 
that 

\vspace{5mm}
\tboxit{Property $\star$}{
for all $y_x\in Y$ and $z_x\in Z$ that lie over $x$ (i.e. $y_x\cap A=z_x\cap A=x$) we get
$$B/y_x \not\cong_{A/x} C/z_x$$
}
\vspace{5mm}

Fix an $x\in X$ with the above property and choose $y_x\in Y$ lying over $x$ (this can be done, since the spectral map 
$Y\rightarrow X$ is a surjective one, see for instance \cite{raphael} Lemma 1.14). Now consider $\p:=y_x\cap B(A) \in \spec B(A)$ then 
$\p B\in \spec B$ and $\p C\in \spec C$  (by Raphael's Lemma) that lie over $x$ and so by Property $\star$
$$B/\p B \not\cong_{A/x} C/\p C$$
\end{proof}

\begin{defi}
Let $f:X\rightarrow Y$ be a function between topological spaces $X$ and $Y$. This function will be called a \emph{near open} (or 
\emph{near-open}) function (German: \emph{fast offene} Abbildung) iff
for all nonempty opens set $U\subset X$ there exists a nonempty open set $V\subset Y$ such that $V\subset f(U)$
\end{defi}

\begin{exam}\indent\par
\begin{enumerate}
\item Let $\R$ be the real numbers endowed with the usual Euclidean topology. Let $f:\R \rightarrow \R$ be defined by $f(x)=x^2$. 
Then this function is a continuous function that is near open however it is not open , because for instance $f((-1,1))=[0,1)$.
\item As will be seen in Theorem \ref{nearopen}, if $A$ is a von Neumann regular ring that is not Baer, then the canonical map
$\spec B(A) \rightarrow \spec A$ is a continuous near open map between Stone spaces that is not open.
\end{enumerate}
\end{exam}

\begin{lemma} \label{vNr+overring}
Let $A$ be a von Neumann regular ring and let $B$ be an overring of $A$. Set 
$$\phi:\spec B \rightarrow \spec A \qquad \p\mapsto \p\cap A$$
Then for any $a\in A$ we have the identity
$$\phi(D_B(a))=D_A(a)$$
\end{lemma}

\begin{proof}
Suppose $a\in A$.

"$\subset$"
Let $\p\in \spec B$ and suppose $a\not\in \p$, then clearly $a\not\in \p\cap A$. In other words $\phi(\p)\in D_A(a)$.

\vspace{5mm}
"$\supset$" Let $\mathfrak q\in \spec A$ and let $a\not\in\mathfrak q$, then by \cite{raphael} Lemma 1.14 there exists a $\p\in\spec B$
such that $\phi(\p)=\mathfrak q$. If $a\in\p$ then $a\in\p\cap A=\mathfrak q=\phi(\p)$ is a contradiction, thus $a\not\in \p$.
So there is a $\p\in D_B(a)$ such that $\phi(\p)=\mathfrak q$.
\end{proof}

\begin{theorem}\label{nearopen}
Let $A$ be a von Neumann regular ring, then the canonical map
$$\phi: \spec B(A) \longrightarrow \spec A$$
is a near open surjection. 
Moreover $\phi$ is open iff $A=B(A)$ (i.e. $A$ is Baer).
\end{theorem}

\begin{proof}
Suppose $U\subset \spec B(A)$ is a nonempty open set. Without loss of 
generality we may assume 
$$U=D_{B(A)}(x)$$ 
for some $x\in B(A)\wo{0}$. 

Now because $B(A)$ is a ring of quotients of $A$ (see for instance the last paragraph of \cite{roqrof} p.8) there exists a $y\in A$ such 
that $xy\in A\wo{0}$ (this is because $A$ is semiprime and commutative, see \cite{roqrof} Theorem following Lemma 1.5). We also then have
$$D_{B(A)}(x)\supset D_{B(A)}(xy)$$
Using the above equation and the preceeding Lemma we obtain
$$\phi(D_{B(A)}(x))\supset D_A(xy)$$
and therefore $\phi$ is near open. $\phi$ is a surjection because of \cite{raphael} Lemma 1.14.

\vspace{5mm}
\noindent Now we prove the last statement of the Theorem, the proof that follows is by Niels Schwartz.

\vspace{5mm}
If $A$ is Baer then $A=B(A)$ and so $\phi$ is a homeomorphism, thus an open map. If $A$ is not Baer then $\spec A$ is not extremally disconnected 
(see Prop. 2.1 \cite{Mewborn}), suppose then that $\phi$ is open. Since $\spec A$ is not extremally disconnected, there exists 
an open set $U\subset \spec A$ such that $\overline U$ (i.e. the topological closure of $U$ in $\spec A$) is not open in $\spec A$. 
Because $\spec B(A)$ is extremally disconnected $\overline{\phi^{-1}(U)}$  (closure in $\spec B(A)$) is clopen, but because $\phi$ is 
a continuous surjection, $\spec B(A)$ is compact and $\spec A$ is Hausdorff we the following result from basic general topology
$$\phi(\overline{\phi^{-1}(U)})=\overline U\subset \spec A$$
And because we assumed $\phi$ is open, the above equation implies that $\overline U$ is open, which is a contradiction.
\\
\end{proof}

\begin{theorem}
Let $A$ be a real regular ring, then $A$ has no 
unique real closure $*$ iff there exists an $x\in A$ such that 
$$[\supp_A P(x)\cap\supp_A P(-x)]^\circ \neq \emptyset$$ 
\end{theorem}

\begin{proof}
"$\Leftarrow$" Theorem 14 of \cite{Capco2} states the 
same thing as this Proposition, however it was assumed there that $A$ is Baer and no mention of near openness is made. 
However in the sufficiency condition of the said Theorem there was no implementation of $A$ being Baer. Thus we need only prove the
necessity for this Proposition.

\vspace{5mm}
"$\Rightarrow$" We almost use the same method of proof as seen in Theorem 14 \cite{Capco2}. Let $C_1,C_2$ be two real closure $*$ of $A$
such that they are not $A$-isomorphic. 

Throughout the proof let $i=1,2$. By Lemma \ref{rcs_of_BHull} we may regard $B(A)$ as a subporing of $C_i$ and denote 
$$\nu_i : \spec C_i \stackrel{\sim}{\longrightarrow} \spec B(A)$$
to be the canonical spectral map (it is a homeomorphism by Raphael's Lemma). 

Now, by Lemma \ref{Baerprime}, there exists a $\p\in B(A)$ such that 
$$C_1/\p C_1 \not\cong_{A/\p\cap A} C_2/\p C_2$$
We observe that $C_i/\p C_i$ is a real closed field (as $C_i$ is real closed $*$, therefore has factor fields that are real closed. 
See \cite{Capco} Theorem 15), and is algebraic over the field $A/\p\cap A$. Thus there are $\alpha_1,\alpha_2\in\sper A$ such that
$$\supp(\alpha_1)=\supp(\alpha_2)=\p\cap A$$
and 
$$\rho(\alpha_i)\cong_{A/\p\cap A} C_i/\p C_i \qquad i=1,2$$

We also have the following commutative diagram of topological (spectral) spaces
$$
\begindc{\commdiag}[1]
\obj(0,60)[A0]{$\sper A$}
\obj(20,50)[A]{$\alpha_i$}
\obj(100,60)[B0]{$\spec B(A)$}
\obj(70,50)[B]{$\p$}
\obj(0,-20)[C0]{$\spec A$}
\obj(20,10)[C]{$\p\cap A$}
\obj(150,30){$i=1,2$}
\mor{B0}{A0}{$\phi_i$}[-1,0]
\mor{B}{A}{}[1,4]
\mor{B0}{C0}{$\psi$}
\mor{B}{C}{}[1,4]
\mor{A0}{C0}{$\supp_A$}[-1,0]
\mor{A}{C}{}[1,4]
\enddc
$$
where 
$$\phi_i := \mu_i\circ\supp_{C_i}^{-1}\circ\nu_i^{-1}$$
with
$$\mu_i : \sper C_i \rightarrow \sper A\qquad \alpha\mapsto \alpha\cap A$$
Note that $\displaystyle\supp_{C_i}$ and $\nu_i$ are homeomorphisms ($C_i$ is a real closed ring too, see \cite{Capco}), therefore $\phi_i$ 
is indeed well-defined.

Now let $x\in \alpha_1\backslash\alpha_2$ then 
$$\p\in\supp_A P(x)\cap \supp_A P(-x)$$
and 
$$\supp_A P(x) \supset \psi\phi_1^{-1}P(x)$$
$$\supp_A P(-x) \supset \psi\phi_2^{-1}P(-x)$$
so
$$\supp_A P(x)\cap\supp_A P(-x) \supset \psi\phi_1^{-1}P(x)\cap \psi\phi_2^{-1}P(-x)\supset \psi(\phi_1^{-1}P(x)\cap \phi_2^{-1}P(-x))$$
but 
$$\p\in \phi_1^{-1}P(x)\cap \phi_2^{-1}P(-x)$$ 
because
$$\phi_1(\p)=\alpha_1\in P(x)$$
and 
$$\phi_2(\p)=\alpha_2\in P(-x)$$
Therefore $\phi_1^{-1}P(x)\cap \phi_2^{-1}P(-x)$ is a nonempty open set in $\spec B(A)$,

Now by Theorem \ref{nearopen} $\psi$ is near open, therefore there exists a nonempty open set $U\subset \spec A$ such that 
$$U\subset\psi(\phi_1^{-1}P(x)\cap \phi_2^{-1}P(-x))$$
Hence 
$$[\supp_A P(x)\cap\supp_A P(-x)]^\circ \neq \emptyset$$
\end{proof}

\begin{defi}
Let $A$ be a poring. By a section of $\supp:\spec A\rightarrow \sper A$, we mean a map 
$$s:\spec A\rightarrow \spec A$$
such that $\supp\circ s =\id_{\spec A}$, where for any set $X$ by $\id_X$ we mean the identity map 
$$\id_X : X\longrightarrow X \qquad x\mapsto x$$
\end{defi}

\begin{theorem}\label{classify_Baer}
Let $A$ be a real Baer regular ring, then there is a one to one correspondence between the set of all real closure $*$ 
of $A$ identified up to $A$-isomorphisms and the set of continuous sections of $\supp_A$.
\end{theorem}

\begin{proof}
Set 
$$\mathcal S := \{s: \spec A \rightarrow \sper A \sep s \textrm{ is continuous and } \supp\circ s =\id_{\spec A} \}$$
and $\mathcal C := \mathcal D/\cong_A$  

\comment{
where
$$\mathcal D:= \{C \sep C \textrm{ is a real closure} * \textrm{ of } A\}$$
and for $C_1,C_2\in \mathcal D$  
$$C_1\sim C_2 \Leftrightarrow C_1\cong_A C_2$$
}

We now attempt to define a bijection $\Phi: \mathcal S \rightarrow \mathcal C$. Let $s\in\mathcal S$, since $s$ is a section 
of $\supp$ (i.e. $\supp\circ s=\id_{\spec A}$) we know then that $A$ can be considered as a subring
of 
$$B:=\prod_{\p\in \spec A} \rho(s(\p))$$
$B$ is a real closed ring (see Remark 1 \cite{Capco}), and therefore $\supp_B$ is a homeomorphism. We thus have the following commutative diagram of 
spectral spaces
$$
\begindc{\commdiag}[1]
\obj(0,50)[A]{$\sper A$}
\obj(0,0)[B]{$\spec A$}
\obj(100,50)[C]{$\sper B \cong \spec B$}
\mor{A}{B}{$\supp$}[-1,0]
\mor{C}{A}{$\phi$}[-1,0]
\mor(90,48)(10,0){$\psi$}
\enddc
$$
where $\phi$ and $\psi$ are canonical maps. 

Now for any set $Z\subset \spec B$, define , as is usual in algebraic geometry, 
$$I_B(Z):=\bigcap_{\p \in Z} \p$$
Define $X:=s(\spec A)$ and observe then that 
\begin{eqnarray*}
I_B(\phi^{-1}(X))\cap A & = & \bigcap_{\p\in\phi^{-1}(X)} \p\cap A = \bigcap_{\p\in\phi^{-1}(X)} \psi(\p)\\
& = & \bigcap_{\p\in\phi^{-1}(X)} \supp(\phi(\p)) = \bigcap_{\mathfrak q\in\supp\circ\phi(\phi^{-1}(X))} \mathfrak q \\
& = &  \bigcap_{\mathfrak q \in \supp(X)} \mathfrak q = \brac{0}
\end{eqnarray*}
the last row of the equation is because $\phi$ is surjective and that 
$\supp(X)=\supp(s(\spec A))=\spec A$. 

We may therefore, by Zorn's Lemma, choose an ideal $I\unlhd B$ such that $I_B(\phi^{-1}(X))\subset I$ and 
$$A\lhrarrow B\longrightarrow B/I$$
is an essential extension of $A$. Set $Y:=\spec B/I\cong \sper B/I$ (Because $B/I$ is real closed, see 
\cite{Capco} Remark 1), we then have 
the following commutative diagram 
$$
\begindc{\commdiag}[1]
\obj(0,50)[A]{$\sper A$}
\obj(100,50)[B]{$Y$}
\obj(0,0)[C]{$\spec A$}
\mor{B}{A}{$\pi$}[-1,0]
\mor{A}{C}{$\supp$}[-1,0]
\mor{B}{C}{$\gamma$}[-1,0]
\enddc
$$
where $\pi$ and $\gamma$ ($\gamma$ being a homeomorphism by Raphael's Lemma) are canonical maps. Define
$$s':\spec A\longrightarrow X \qquad s'(\p):=s(\p) \,\forall \p\in\spec A $$
We now claim $\dots$

\vspace{5mm}
\noindent\underline{Claim 1: $s'\circ\supp|X = \id_X$} \\
We know 
$$s'\circ\supp\circ s'=s'$$
and we know that $s'$ is bijective (as $s$ is a section and therefore injective). 
So we may compose the right side by $s'^{-1}$ and we get the desired identity!

\vspace{5mm}
\noindent\underline{Claim 2: $\pi(Y)=X$}\\
Since $I\supset I_B(\phi^{-1}(X))$ and since we know that $\phi^{-1}(X)$ is closed 
(this is because $\phi$ is continuous and $s$ is a continuous map between a compact space and a Hausdorff space, 
and so $X=s(\spec A)$ and $\phi^{-1}(X)$ are closed)
in $\spec B$, we then know that 
$$Y\cong V_B(I) \subset V_B(I_B(\phi^{-1}(X)))=\phi^{-1}(X)$$
This imples that 
$$\pi(Y)=\phi(V_B(I))\subset \phi(\phi^{-1}(X))\subset X$$
therefore $\pi(Y)\subset X$
and so by Claim 1 we get
$$s\circ \gamma=s\circ\supp\circ\pi=\pi$$

In other words we have the commutative diagram 
$$
\begindc{\commdiag}[1]
\obj(0,50)[A]{$\sper A$}
\obj(100,50)[B]{$Y$}
\obj(0,0)[C]{$\spec A$}
\mor{B}{A}{$\pi$}[-1,0]
\mor{C}{A}{$s$}
\mor{B}{C}{$\gamma$}
\enddc
$$
but
$s\circ\gamma(Y)=\pi(Y)=s(\spec A)=X$ (because $\gamma$ is a homeomorphism and thus a surjection).

\vspace{5mm}
Now define 
$$ \Phi(s) := \ic(A,B/I)/\cong_A$$
we need yet to show that $\Phi$ defined in this way for any $s\in \mathcal S$ is $\dots$

\vspace{5mm}
\noindent\underline{Claim 3: well-defined} \\
In other words we need to show that for $s\in \mathcal S$, $\Phi(s)$ is in $\mathcal C$ and is independent of the choice of $I$ 
(as constructed above). Let $B$ and $I\unlhd B$ be as constructed above. Because $B/I$ is a von Neumann regular ring that is essential over 
the Baer ring $A$, $B/I$ is Baer and real closed (by Raphael's Lemma and Remark 1 in \cite{Capco}). Therefore $B/I$ is a real  closed 
$*$ ring (by \cite{Capco} Theorem 15). 
And so by \cite{Capco2} Proposition 6 $\ic(A,B/I)\in \mathcal D$. This proves that $\Phi(s)\in \mathcal C$. 

Now suppose that $I_1,I_2$ are two ideals in $B$ such that 
$$I_1,I_2 \supset I_B(\phi^{-1}(X)$$ 
and such that $B/I_1, B/I_2$ are essential extensions of $A$. We show that 
$$\ic(A,B/I_1)\cong_A \ic(A,B/I_1)$$
Let $i=1,2$ and define $C_i:=\ic(A,B/I_i)$. We then have the following commutative diagram of porings
$$
\begindc{\commdiag}[1]
\obj(0,50)[A]{$A$}
\obj(100,50)[B]{$B/I_i$}
\obj(100,0)[C]{$C_i$}
\mor{A}{B}{}[1,3]
\mor{A}{C}{}[1,3]
\mor{C}{B}{}[1,3]
\enddc
$$
with all the maps being canonical injections (whose spectral maps on their prime spectra are all homeomorphic). Now suppose that 
$\p\in \spec A$ then there is a unique $\p_i\in \spec B/I_i$ such that $\p_i\cap A=\p$ (in fact $\p_i=\p B/I_i$ by Raphael's Lemma). Now
accroding to the commutative diagram in Claim 2, we have the following commutative diagram of spectral spaces
$$
\begindc{\commdiag}[1]
\obj(0,50)[A]{$\sper A$}
\obj(0,0)[B]{$\spec A$}
\obj(120,50)[C]{$\sper B/I_i \cong \spec B/I_i$}
\obj(200,25){$i=1,2$}
\mor{C}{A}{$\pi_i$}[-1,0]
\mor{B}{A}{$s$}
\mor(110,48)(10,0){$\gamma_i$}
\enddc
$$
where $\pi_i,\gamma_i$ are canonical maps. Therefore $s\gamma_i(\p_i)=\pi(\p_i)=s(\p)$ and so because $C_i/\p C_i$ and 
$B/\p_i$ are real closed fields we obtain
$$C_i/\p C_i =\ic(A/\p, B/\p_i)\cong_{A/\p} \rho(s(\p))$$
and this is valid for all $\p\in \spec A$. Thus by Theorem 8 of \cite{Capco2}
$$C_1\cong_A C_2$$

\vspace{5mm}
\noindent\underline{Claim 4: injective} \\
Let $s,t\in\mathcal S$. Suppose also that $\Phi(s)=\Phi(t)$. Let $C\in \mathcal D$ such that 
$$\Phi(s)=\Phi(t)=C/\cong_A$$
as we have seen in Claim 3, we know that for all $\p\in \spec A$  one has
$$\rho(s(\p)) \cong_{A/\p} C/\p C \cong_{A/\p} \rho(t(\p))$$
thus one concludes at once that for all $\p\in\spec A$ one has $s(\p)=t(\p)$ and therefore $s=t$

\vspace{5mm}
\noindent\underline{Claim 4: surjective} \\
Let $C\in\mathcal D$, one then has the following commutative diagram of spectral spaces
$$
\begindc{\commdiag}[1]
\obj(0,50)[A]{$\sper A$}
\obj(0,0)[B]{$\spec A$}
\obj(100,50)[C]{$\sper C \cong \spec C$}
\mor{A}{B}{$\supp$}[-1,0]
\mor{C}{A}{$\pi$}[-1,0]
\mor(90,48)(10,0){$\gamma$}
\enddc
$$
where $\gamma$ and $\pi$ are canonical maps. So here, for any $\mathfrak q\in \spec C$ (because $C$ is real closed) we have the identity
$$\rho(\pi(\mathfrak q)) \cong_{A/\mathfrak q\cap A} C/\mathfrak q$$

\comment{
Now define
$$X:=\{\pi(\p C) \sep \p\in\spec A\} = \{\pi(\mathfrak q) \sep \mathfrak q\in\spec C\}$$
(as $\gamma$ is is a homeomorphism by Raphael's Lemma). 
}

Now define 
$$s:\spec A \rightarrow \sper A \qquad s(\p):=\pi(\gamma^{-1}(\p)) \,\forall \p\in\spec A$$
We show first that $s\in\mathcal S$. For all $\p\in\spec A$ we get
$$\supp\circ s(\p)=\supp\pi(\gamma^{-1}(\p))=\gamma\gamma^{-1}(\p)=\p$$
Thus $\supp\circ s = \id_{\spec A}$ (i.e. $s$ is indeed a section of $\supp$). Because both $\pi$ and $\gamma^{-1}$ are continuous maps
we see then that $s$ is a continous map. 
\comment{
Also by the definition of $X$ we have
$$s(\spec A)=\pi\gamma ^{-1}(\spec A) = \pi(\spec C)= X$$
and so $s$ is a surjective continous section. 
}

We now show that $\Phi(s)=C/\cong_A$. Let $C'\in\mathcal D$ such that $C'/\cong_A = \Phi(s)$. But from Claim 3 we have seen that for 
any $\p \in \spec A$ we have 
$$C'/\p C' \cong_{A/\p} \rho(s(\p))=\rho(\pi\gamma^{-1}(\p))\cong_{A/\p} C/\p C$$
One then uses Theorem 8 of \cite{Capco2} to claim that $C\cong_A C'$.
\end{proof}

\begin{prop} Let $A$ be a real von Neumann regular ring, then a section of $\supp_A$ is a homeomorphism onto its image iff it is 
continuous.
\end{prop}
\begin{proof}
The proof is quite straightforward. One side of the equivalence is trivial. Because $\spec A$ and $\sper A$ are compact and Hausdorff
then the image of a continuous section of $\supp$ is closed compact and Hausdorff in $\sper A$. The same reasoning tells us that
the section brings closed sets to closed set in the image (because it is a continuous map from a compact space to a Hausdorff space). The 
section being injective is thus a homeomorphism onto its image.
\end{proof}

Let $A$ be a real Baer regular ring, and set $X:=\spec A$ and $Y:=\sper A$. 
In \cite{MN} Chapter I there is a beautiful treatise on the different topologies that the set of continuous functions from $X$ to $Y$ may have
(and by which practical application may be applied on these topologies). We shall
denote the set of continuous functions from $X$ and $Y$ as $C(X,Y)$ for now. $C(X,Y)$ may have the so called \emph{point convergence 
topology} which is simply
the topology relative to the Tychonoff product topology of $Y^X$. A finer topology would be the 
\emph{compact-open topology} (see \cite{MN} p.4). 
As is shown in Theorem 1.1.3 of \cite{MN} most reasonable topologies of $C(X,Y)$ contain the point convergence topology.
So if we show that a subset of $C(X,Y)$ is closed with respect to the point convergence topology
then it is automatically closed in these other topologies of $C(X,Y)$ (namely those induced by \emph{closed networks} on $X$,
for terminologies and further reading the reader is advised to consult \cite{MN} Chapter I).

Below is a Lemma that is proven by K.P. Hart (with a bit of rewording by me) in the sci.math newsgroup during one of our discussion 
regarding the set of continuous sections of a continuous map.

\begin{lemma}(K.P. Hart, 12.2007) Given a surjective continuous function between T1 topological spaces, say $\pi :Y \rightarrow X$, the set of continuous sections
of $\pi$ is closed in $C(X,Y)$ (i.e. set of continuous functions from $X$ to $Y$) with the point convergence topology.
\end{lemma}
\begin{proof}
Let 
$$F_x :=\{f \in Y^X  \sep f(x)\in\pi^{-1}(x)\}$$
then this set is obviously closed (with the point convergence topology) in $Y^X$ and the set of continuous sections 
of $\pi$ can be written as the intersection
$$\bigcap_{x\in X} F_x \cap C(X,Y)$$
and this is also obviously closed relative to $C(X,Y)$.
\end{proof}

\begin{cor}
Let $A$ be a real Baer von Neumann regular ring, then the set of real closure $*$ of $A$ identified upto $A$-isomorphism form a 
Hausdorff topological space and can be identified as a closed subspace of $C(X,Y)$ with the point convergence topology (and thus also 
in other finer topologies induced by \emph{closed networks} on $X$ as defined in \cite{MN} p.3, this fact is due to
Theorem 1.1.3 of \cite{MN})
\end{cor}
\begin{proof}
Because of Theorem \ref{classify_Baer}, we may identify the set of real closure $*$ of $A$ with the set of continuous sections of 
$\supp_A$. Set $X:=\spec A$ and $Y:=\sper A$ and write $C(X,Y)$ to be the set of continuous functions from $X$ to $Y$ and use 
the above Lemma substituting $\pi$ with $\supp_A$.
\end{proof}

\comment{
Note that the space defined in the above Proposition is, in most cases, not compact in the point convergence topology. 
Let $X$ and $Y$ be as in the above proposition. $C(X,Y)$ with the point convergence topology 
is dense in $Y^X$ and the two sets
are hardly the same, so we know that $C(X,Y)$ is not compact (since $Y^X$ is Hausdorff and so $C(X,Y)$ need to at least be closed to be compact!).
And any compact subset of $C(X,Y)$ is also compact in $Y^X$ with the product topology. It therefore must be closed in $Y^X$. 
Now suppose the space defined in the above proposition is compact, 
by the proof of the Proposition above $\bigcap F_x$, for cases by which $\supp_A$ has a non-continuous section, have elements outside 
of $C(X,Y)$ which is a closure point of $C(X,Y)$. Thus $\bigcap F_x \cap C(X,Y)$ has a closure point that it doesn't contain, a contradaction.
Thus in general if $\supp_A$ has a section that is non-continuous, we have for any topology finer than the 
point convergence topology of $C(X,Y)$, the set of continuous sections of $\supp_A$ is not compact in $C(X,Y)$. In particular
the set of continuous sections of $\supp_A$ is not compact in the compact open topology if $\supp_A$ contains sections that are 
non-continuous.

\vspace{5mm}
}

During the investigation of von Neumann regular rings, I made many use of the Baer hull of the ring. It was therefore natural to ask the
question whether the Baer hull and the complete ring of quotients of such rings coincide. The example below shows that one may indeed 
have a Baer von Neumann regular ring that is not rationally complete.

\begin{exam}
Let $K$ be a real field (say $\R$). Also define a ring 
$$ R:=\prod_{x\in K} K_x \qquad K_x:=K \,\,\forall x \in K$$
with canonical (componentwise) addition and multiplication. We may also from now on regard $K$ as a subring $R$ by taking the
canonical monomorphism 
$$K\lhrarrow R \qquad k \mapsto \{k_x | x\in K, k_x=k \}$$

We now define a subring of $R$ 
$$A:=\{\sum_{i=1}^n e_ix_i \sep e_i\in E(R), n\in\N\} $$
We shall now give some facts regarding $A$ with sketches of their proof

\vspace{5mm}
\noindent\underline{Claim 1} For any $a\in A$ we claim that we may write $a$ as 
$$a=\sum_{i=1}^n e_ix_i$$
with $x_i\in K$ and $e_i\in E(R)$ for $i=1,\dots,n$ and the $e_i$'s satisfy the fact that they have pairwise disjoint supports
. In other words 
$$\{x\in K \sep e_i(x)\neq 0\}\cap \{x\in K \sep e_i(x)\neq 0\}=\emptyset \qquad i,j=1,\dots,n\quad i\neq j$$

To show this, we first write $a$ as $\sum_{i=1}^m f_iy_i$ for some $y_i\in K$ and $f_i\in E(R)$ (by definition of $A$). Now
we define $\mathcal S$ to be the powerset of $\{1,\dots,m\}$ without the emptyset and for any $S\in\mathcal S$ set
$$e_S :=\prod_{j\in S} f_j \prod_{k\not\in S} (1-f_k) $$
and
$$X_S :=\{x\in X \sep f_S(x) \neq 0\}$$
Then one shows that for any $S,T\in \mathcal S$ such that $S\neq T$ we get $S\neq T$ and we have the identity
$$a=\sum_{i=1}^m f_i y_i =\sum_{S\in\mathcal S} e_S\sum_{j\in S} y_j$$
Thus we may write $a$ as a linear combination (with $K$ as the scalar) of $2^n-1$ idempotents with disjoint support.

\vspace{5mm}\noindent \underline{Claim 2} One checks that $A$ is a proper subring of $R$. To check that $A$ is strictly
contained in $R$, one need to only show that the element $r\in R$ defined by $r(x)=x$ is not in $A$. To do this we note
a fact that $r$ can never be written as linear combination of idempotents of $R$ with disjoint supports, and then we make 
use of Claim 1. 

\vspace{5mm}\noindent\underline{Claim 3} We now claim that $A$ is in fact von Neumann regular. Let $a\in A\wo{0}$, then we may 
write $a$ as 
$$a=\sum_{i=1}^n e_i x_i \qquad x_i\in K\wo{0}, e_i\in E(R)\wo{0}$$
with $e_i$'s having pairwise disjoint supports. Then define $a'\in A$ by $a'=\sum e_ix_i^{-1}$, one easily sees that
$a'$ is the quasi-inverse of $a$, i.e. $a^2a'=a$. Because $a$ was an arbitrary nonzero element of $A$, we have proven that
any element of $A$ has a quasi-inverse and so the ring is von Neumann regular.

\vspace{5mm}\noindent\underline{Claim 4} $R$ is a rational extension of $A$ and $A$ is a Baer proper subring of $R$. $R$ is 
obviously a rationally complete ring (its the product of fields). And if $r\in R\wo 0$ then one can easily multiply it by 
an idempotent with finite support to have an element in $A$. So $R$ is a rational extension of $A$ which is rationally complete
and thus the complete ring of quotients of $A$ is $R$. $A$ also has all the idempotents of $R$, thus $A$ is Baer by Mewborn's
Proposition (see \cite{Capco} Proposition 2).
\begin{flushright}$\blacksquare$\end{flushright}
\end{exam}

\begin{nota}
Let $A$ be a poring and $\alpha\in \sper A$, then we write $A(\alpha)$ to mean the real field $\qf{A/\supp(\alpha)}$ with
the canonical partial ordering corresponding to $\alpha$ (i.e. $\alpha/\supp(\alpha)\subset A(\alpha)^+$). The real closed field
(upto $A/\supp(\alpha)$-isomorphism) which is a real field extension of $A(\alpha)$ will then be denoted as $\rho(A(\alpha))$. We 
formerly used $\rho(\alpha)$ to denote this, but there is a good reason why we use $\rho(A(\alpha))$ instead. Firstly $\rho$ was the symbol
first used to mean the real closure (in the sense of Niels Schwartz) functor, and $\rho(A(\alpha))$ is indeed the real closure of 
$A(\alpha)$. Therefore we reduce confusion here (since $\rho(\alpha)$ used previously had nothing to do with the real closure functor 
$\rho$). Secondly, sometimes it is important for us to specify the ring involved and $\rho(A(\alpha))$ does show us that we are dealing 
with the poring $A$. So, we shall henceforth make use of this notation.
\end{nota}

\begin{theorem}\label{sperB}
Let $A$ be a real von Neumann regular ring and consider the pullback
$$
\begindc{\commdiag}[1]
\obj(0,50)[A]{$\sper A \times_{\spec A} \spec B(A)$}
\obj(100,50)[B]{$\sper A$}
\obj(0,0)[C]{$\spec B(A)$}
\obj(100,0)[D]{$\spec A$}
\mor{A}{B}{}
\mor{B}{D}{$\supp_A$}
\mor{A}{C}{}
\mor{C}{D}{$\phi_p$}[-1,0]
\enddc
$$
with $\phi_p$ being the canonical map (i.e. $\phi_p(\tilde\p):=\tilde\p\cap A$ for all $\tilde\p\in\spec B(A)$). 
It turns out then that the fiber product 
$$\sper A \times_{\spec A} \spec B(A)$$
is (canonically) homeomorphic to $\sper B(A)$.
\end{theorem}
\begin{proof}
Abbreviate $B:=B(A)$, set $X:=\sper A \times_{\spec A} \spec B$ and name the projection of the pullback by
$$\pi_A : X\longrightarrow \sper A$$
and
$$\pi_B : X\longrightarrow \spec B$$
Then we have the following commutative diagram in spectral spaces
$$
\begindc{\commdiag}[1]
\obj(0,50)[A]{$\sper B$}
\obj(100,50)[B]{$\sper A$}
\obj(0,0)[C]{$\spec B$}
\obj(100,0)[D]{$\spec A$}
\mor{A}{B}{$\phi_r$}
\mor{B}{D}{$\supp_A$}
\mor{A}{C}{$\supp_B$}[-1,0]
\mor{C}{D}{$\phi_p$}[-1,0]
\enddc
$$
where $\phi_r$ is the canonical map. By the universal property of the pullback there is a unique continuous map 
$\psi: \sper B \rightarrow X$ such that the diagram below commutes
$$
\begindc{\commdiag}[1]
\obj(50,50)[A]{$X$}
\obj(120,50)[B]{$\sper A$}
\obj(50,0)[C]{$\spec B$}
\obj(120,0)[D]{$\spec A$}
\obj(-20,80)[E]{$\sper B$}
\mor{A}{B}{$\pi_A$}[-1,0]
\mor{B}{D}{$\supp_A$}
\mor{A}{C}{$\pi_B$}
\mor{C}{D}{$\phi_p$}[-1,0]
\mor{E}{A}{$\psi$}[-1,0]
\mor{E}{B}{$\phi_r$}
\mor{E}{C}{$\supp_B$}[-1,0]
\enddc
$$
This Theorem claims that $\psi$ is in fact a homeomorphism.

Observe that because $B$ is an integral poring extension of $A$, one has for any $\tilde\alpha\in\sper B$ the identity 
$$\rho(A(\alpha))\cong_{A/\supp(\alpha)} \rho(B(\tilde \alpha))$$
(see also Lemma 2(i) in \cite{Capco2}) where $\alpha:=\tilde\alpha\cap A=\phi_r(\tilde\alpha)$. Now we show that $\dots$

\vspace{5mm}
\noindent\underline{$\psi$ is injective}\\
Let $\tilde\alpha,\tilde\beta\in\sper B$ such that $\psi(\tilde\alpha)=\psi(\tilde\beta)=:x$ for some $x\in X$. Then
$$\pi_A(x)=\tilde\alpha\cap A =\tilde\beta \cap A =:\alpha \in \sper A$$
for some $\alpha\in\sper A$, this implies that
$$\rho(B(\tilde \beta))\cong_{A/\supp(\alpha)}\rho(A(\alpha))\cong_{A/\supp(\alpha)} \rho(B(\tilde \alpha))$$
Also
$$\pi_B(x)=\supp_B(\tilde\alpha)=\supp_B(\tilde\beta)=\tilde\p\in\spec B$$
for some $\tilde\p\in\spec B$. But the prime cone $\tilde\alpha$ of $B$ can be regarded also as the pair 
$$(\rho(B(\tilde\alpha)),\supp_B(\tilde\alpha))=(\rho(\alpha),\tilde\p)=(\rho(B(\tilde\beta)),\supp_B(\tilde\beta))$$
see for instance \cite{KS} \S3 or Proposition 1.3 in \cite{ABR} so in fact $\tilde\alpha=\tilde\beta$.

\vspace{5mm}
\noindent\underline{$\psi$ is surjective}\\
We may regard the elements of $X$ as pairs of the form $(\alpha,\tilde\p)\in\sper A\times \spec B$ such that 
$\supp_A(\alpha)=\tilde\p\cap A=\phi_p(\tilde\p)$. Thus let $(\alpha,\tilde\p)\in X$ and let the prime
cone of $B$ associated to the pair $(\rho(A(\alpha)),\tilde\p)$ be denoted by $\tilde\alpha$, in fact specifically
$$\tilde\alpha=\{b\in B\sep b\mod\tilde\p \in \rho(A(\alpha))^+\}$$
(see for instance remark in \cite{KS} after Satz 1, p.108). Then 
$\supp_B(\tilde\alpha)=\tilde\p$ and $\tilde\alpha\cap A=\phi_r(\tilde\alpha)=\alpha$ and therefore (by the definition 
of $X$) we get $\psi(\tilde\alpha)=(\alpha,\tilde\p)$. 

\vspace{5mm}
It is easy to see that $X$ is also a Stone space, therefore we have a continuous bijection $\psi$ between a compact space 
$\sper B$ and a Hausdorff space $X$. This bijection is therefore also a closed map and thus a homeomorphism.
\end{proof}

\begin{defi} Let $X$ be a topological space with topology $\mathcal T$ then
\begin{enumerate}
\item An \emph{open filter}, $\mathcal U$ on $X$ is a subset of $\mathcal T$ which is also a filter (with the usual containment as partial ordering)
\item Similarly one defines an \emph{open ultrafilter} on $X$
\end{enumerate}
\end{defi}

Below is a construction of absolutes of Hausdorff space as implemented by Porter and Woods in \cite{pw1} \S6.6 and in \cite{pw2} \S3.1.

\begin{cons}(Iliadis absolutes) Let $X$ be a Hausdorff space with topology $\mathcal T$. It is shown in \cite{pw1} \S6.6(d) that if 
$\mathcal U$ is an open ultrafilter on $X$ one has 
$$\bigcap_{U\in\mathcal{U}} U\neq \emptyset \Leftrightarrow \exists\,x\in X\st \bigcap_{U\in\mathcal{U}}=\{x\}$$

The \emph{Gleason space} of $X$, denoted $\theta X$, consists of the set of all open ultrafilters on $X$ equipped with a topology generated
by the open basis consisting of the sets of the form
$$\{\mathcal U\in \theta X \sep U\in\mathcal U\}\qquad U\in \mathcal T$$

The \emph{Iliadis absolute} or \emph{absolute} of $X$ is defined by 
$$\mathcal E X :=\{\mathcal U\in \theta X \sep \bigcap_{U\in\mathcal{U}} U\neq \emptyset\}$$
and it is equipped with the subspace topology of $\theta X$. It is shown in \cite{pw1} \S6.6(e) that $\mathcal E X$ is Stone and extremally
disconnected. 

There is a surjection from $\mathcal E X$ to $X$, which we shall call the \emph{projection of the absolute of $X$} and 
denote it by $\pi_X$ which is defined by 
$$\pi_X : \mathcal E X \lthrarrow X \qquad \pi(\mathcal U) :=\bigcap_{U\in\mathcal{U}}U \quad\forall\mathcal U\in \mathcal E X$$
It is shown in \cite{pw1} \S6.6(e)(6) that $X$ is regular (as topological space) iff $\pi_X$ is continuous. In particular if $X$ is Stone
then $\pi_X$ is a continuous map.
\begin{flushright}$\blacksquare$\end{flushright}
\end{cons}

\begin{defi}
A function $f: X\rightarrow Y$ between two topological spaces is called an \emph{irreducible surjection} iff the function
is continuous, surjective, closed and for any proper closed set $C\subsetneq X$ we have $f(C)\subsetneq Y$.
\end{defi}

The above definition can be found in \cite{pw1} 6.5(a). However, when discussing about a function having the property 
in the above definition we always accompany the word \emph{irreducible} with the word \emph{surjection} in order to avoid confusion 
(because "irreducible" is very frequently used in mathematics and could mean many different things).

\begin{lemma}\label{baer_ired}
Let $A$ be a von Neumann regular ring, then the canonical map $\phi: \spec B(A) \rightarrow \spec A$ is an 
irreducible surjection
\end{lemma}
\begin{proof}
That $\phi$ is continuous and closed is clear (because $\spec A$ is Hausdorff and $\spec B(A)$ is compact), 
it is also clearly surjective (see for instance \cite{raphael} Lemma 1.14). Suppose now that there is a closed set
$C\subsetneq \spec B(A)$ such that $\phi(C)=\spec A$. Without loss of generality we may assume 
$C$ to be of the form $V_{B(A)}(b)$ for some $b\in B(A)\wo{0}$ ($\spec B(A)\backslash C$ is open, so there is a nonempty basic open
set contained in it). Now because $B(A)$ is a rational extension of $A$, there is an $a\in A$ such that $ba\in A\wo{0}$. We
know by Lemma \ref{vNr+overring} that $\phi(V_{B(A)}(ab))=V_A(ab)$ (because we have a regular
ring, we can express $V_{B(A)}(ab)=D_{B(A)}(x)$ for some $x\in A$) and so
$$V_{B(A)}(b)\subset V_{B(A)}(ab) \Rightarrow \spec A =\phi(V_{B(A)}(b))\subset \phi(V_{B(A)}(ab))=V_A(ab)=\spec A \Rightarrow ab=0$$
which is a contradiction.
\end{proof}

Because the above Lemma only uses the fact that $B(A)$ is a rational extension of $A$, we can use the same proof to show

\begin{cor}\label{vNr_ired}
Let $A$ be a von Neumann regular ring and let $B$ be a rational extension of $A$, then the canonical map 
$$\spec B \longrightarrow \spec A$$
is an irreducible surjection.
\end{cor}

\begin{prop} If $A$ is a von Neumann regular ring, then there is a homeomorphism 
$$\psi:\mathcal E\spec A \rightarrow \spec B(A)$$ 
such that $\phi\circ \psi =\pi_{\spec A}$, where 
$$\phi:\spec B(A)\rightarrow \spec A$$
is just the canonical map.
\end{prop}
\begin{proof}
This is an immediate consequence of the above Lemma, the fact that $\spec A$ is a regular space (since it is Stone), \cite{pw1} \S6.1(a) and 
\cite{pw1} \S4.8(h)(3).
\end{proof}

So we do see that no direct information of the ring structure of $B(A)$ (for a real regular ring $A$) is necessary to obtain information 
about the topological space
$\spec B(A)$ and $\sper B(A)$, the only information we needed for these topological spaces were those of $\spec A$ and $\sper A$.

Now we try to classify the real closure $*$ of an arbitrary real von Neumann regular ring. One may expect a combination of Lemma 
\ref{rcs_of_BHull} and a modification of Theorem \ref{classify_Baer}, however the result is rather more complicated than just that. 
We may indeed argue that the set of real closure $*$ of a real regular ring, say $A$, is the same as the set of real closure $*$ of its Baer hull 
(from Lemma \ref{rcs_of_BHull}). But we are dealing here with the sets with an equivalence relation that identify the real closure $*$ 
of $A$ upto $A$-isomorphisms. And we are not aware whether $B(A)$-isomorphism and $A$-isomorphism of the real closure $*$ are equivalent.
What has been just discussed is best illustrated by the following Proposition (and its proof).

\begin{prop} Let $A$ be a real regular ring and set $B$ to be the Baer hull of $A$. Define now the following 
\begin{enumerate}
\item $\mathcal S:=\{s:\spec B\rightarrow \sper B \sep s \textrm{ is a continuous section of } \supp_B\}$
\item $\mathcal C:=\{C\sep C \textrm{ is a real closure $*$ of } A\}=\{C\sep C \textrm{ is a real closure $*$ of } B\}$
\item $\pi_r : \sper B \rightarrow \sper A \qquad \tilde\alpha\mapsto \tilde\alpha\cap A$
\item an equivalence relation $\sim$ on $\mathcal S$
$$s\sim t \Leftrightarrow \supp_A(\pi_r(s(\spec B))\cap \pi_r(t(\spec B)))=\spec A\qquad (s,t\in\mathcal S)$$
\end{enumerate}
Then there is a bijection
$$\Phi:\mathcal S/\sim\,\bij \mathcal C/\cong_A$$
\end{prop}
\begin{proof}
Define first $\Phi_B : \mathcal S \rightarrow \mathcal C/\cong_B$ to be the bijection between the continuous sections of $\supp_B$
and the real closure $*$ of $B$ upto $B$-isomorphism as shown in Theorem \ref{classify_Baer}. Now for any $s\in\mathcal S$ 
set $C_s$ to be any chosen ring in $\mathcal C$ such that $\Phi_B(s)=C_s/\cong_B$ (throghout, as we are dealing with $A$-isomorphisms our
proof will be independent of the choice of this $C_s$ for any $s$). 

We now need to first show that $\sim$ is actually an equivalence relation on $\mathcal S$. The only difficult problem 
actually lies on proving transitivity. We claim that 
$$s\sim t \Leftrightarrow C_s\cong_A C_t\qquad (s,t\in\mathcal S)$$
(independent of the choices of $C_s$ and $C_t$) and if we show this then we have also shown that $\sim$ is an equivalence 
relation. Let $s,t\in \mathcal S$ then 
$$
\begin{array}{c}
s\sim t \\
\Updownarrow \\
\forall\p\in\spec A \,\,\exists\tilde\p_s,\tilde\p_t\in\spec B \textrm{ and } \alpha\in\sper A \st \\
\tilde\p_s\cap A=\tilde\p_t\cap A =\supp_A(\alpha)=\p \\
\textrm{ and } s(\tilde\p_s)\cap A = t(\tilde\p_t)\cap A = \alpha \\ 
\Updownarrow \\
\forall\p\in\spec A \,\,\exists\tilde\p_s,\tilde\p_t\in\spec B \textrm{ and } \alpha\in\sper A \st \\
\tilde\p_s\cap A=\tilde\p_t\cap A =\supp_A(\alpha)=\p \\
\textrm{ and } C_s/\tilde\p_sC_s \cong_{A/\p} C_t/\tilde\p_tC_t\cong_{A/\p}\rho(B(s(\tilde\p_s)))
\cong_{A/\p}\rho(B(t(\tilde\p_t)))\cong_{A/\p}\rho(A(\alpha))\\
\Updownarrow \\
C_s\cong_A C_t
\end{array}
$$

Now for any $s\in\mathcal S$ let us denote $\bar s$ to be the image of $s$ in $\mathcal S/\sim$. For such an $s$ we now set
$$\Phi(\bar s) = C_s/\cong_A$$
and we show that $\Phi$ defined in such way is $\dots$

\vspace{5mm}
\noindent\underline{well-defined}. Let $s,t\in \mathcal S$ and $s\sim t$. Then by our previous claim this is equivalent to
$C_s\cong_A C_t$. And thus $\Phi$ is indeed well-defined.

\vspace{5mm}
\noindent\underline{bijective}. Injectivity is almost clear, because if $\Phi(\bar s)=\Phi(\bar t)$ for some $s,t\in \mathcal S$ then
by construction of $\Phi$ we get $C_s\cong_A C_t$ which by our very first claim implies that $s\sim t$. 

Surjectivity is due to the fact that for any real closure $*$ of $A$ say $C$, there is an $s\in\mathcal S$ such that $\Phi_B(s)=C/\cong_B$.
And we thus have $\Phi(s)=C/\cong_A$.
\end{proof}

\begin{ack}
I would like to thank Prof. Niels Schwartz for his most valuable advises.
\end{ack}


\begin{thebibliography}{99}

\bibitem{ABR}
\textbf{C. Andradas, L. Br\"ocker, J.M. Ruiz},
"Constructible Sets in Real Geometry",
Ergebnisse der Mathematik und ihrer Grenzgebiete, 3. Folge, Bd. \textbf{33},
Springer-Verlag 1996

\bibitem{Brum}
\textbf{G.W. Brumfiel},
"Partially Ordered Rings and Semi-Algebraic Geometry",
London Math. Soc. Lecture Note Series \textbf{37},
Cambridge University Press 1979

\bibitem{Capco}
\textbf{J. Capco},
Real Closed Rings and Real closed $*$ Rings, 
arXiv Preprint 06.2007, Online \mbox{\url{http://arxiv.org/abs/0707.2189}},
Accessed 12.12.2007

\bibitem{Capco2}
\textbf{J. Capco},
Uniqueness of real closure $*$ of Baer regular rings, 
arXiv Preprint 08.2007, Online \mbox{\url{http://arxiv.org/abs/0710.0267}}, 
Accessed 12.12.2007

\bibitem{roqrof}
\textbf{N.J. Fine, L. Gillman, J. Lambek}, 
"Rings of Quotients of Rings of Functions",\\
Transcribed and edited into PDF from the original 1966 McGill University Press book \\
(see \mbox{\url{http://tinyurl.com/24unqs}}, Editors: M. Barr, R. Raphael),\\
Online \mbox{\url{http://tinyurl.com/ytw3tj}},
Accessed 24.10.2007

\comment{
\bibitem{Joshi}
\textbf{K.D. Joshi},
"Introduction to General Topology", 
Wiley Eastern Limited, 1983
}

\bibitem{KS}
\textbf{M. Knebusch, C. Scheiderer},
"Einf\"uhrung in die reelle Algebra", Vieweg 1989

\bibitem{lambek}
\textbf{J. Lambek},
"Lectures on Rings and Modules", 
Second Edition, 
Chelsea Publishing Company 1976

\bibitem{MN}
\textbf{R.A. McCoy, I. Ntantu},
"Topological Properties of Spaces of Continuous Functions",
Lecture Notes in Mathematics \textbf{1315}, 
Springer-Verlag 1988

\bibitem{Mewborn}
\textbf{Ancel C. Mewborn},
"Regular Rings and Baer Rings", 
Math. Z. 1971, vol. 121, p. 211-219

\bibitem{pw2}
\textbf{J.R. Porter, R.G. Woods},
"Extensions of Hausdorff Spaces", 
Pacific Journal of Mathematics, 1982, vol. 103, No. 1, p. 111-134

\bibitem{pw1}
\textbf{J.R. Porter, R.G. Woods},
"Extensions and Absolutes of Hausdorff Spaces", 
Springer-Verlag 1988

\bibitem{raphael}
\textbf{R.M. Raphael},
"Algebraic Extensions of Commutative Regular Rings", 
Canad. J. Math. 1970, vol. 22, p. 1133-1155

\bibitem{Storrer}
\textbf{H.H. Storrer},
"Epimorphismen von kommutativen Ringen", 
Comm. Math. Helvetici 1968, vol. 43, p. 378-401

\bibitem{jstone}
\textbf{P.T. Johnstone},
"Stone Spaces", 
Cambridge University Press 1982

\end{thebibliography}
\end{document}